\documentclass{amsart}

\usepackage{amsmath,amsthm}
\usepackage{amssymb}

\usepackage[english,francais]{babel}

\newcommand{\comment}[1]{}

\newtheorem{theorem}{Theorem}
\newtheorem{definition}[theorem]{Definition}

\newtheorem{problem}[theorem]{Problem}

\newtheorem{remark}[theorem]{Remark}

\newcommand{\NN}{\mathbb N}
\newcommand{\ZZ}{\mathbb Z}
\newcommand{\RR}{\mathbb R}

\newcommand{\TT}{\mathbb T}
\newcommand{\TP}{{\mathcal T}^{+}}

\newcommand\PP{\mathcal{P}}

\newcounter{rev}
\newcounter{reb}
\begin{document}
\centerline{}


\selectlanguage{english}
\title{Failure of Wiener's property for positive definite periodic
functions}


\selectlanguage{english}
\author[A. Bonami]{Aline Bonami}
\email{aline.bonami@univ-orleans.fr}
\author[S. Revesz]{Szil\'ard Gy. R\'ev\'esz}
\email{revesz@renyi.hu}

\address[A. Bonami]{F\'ed\'eration Denis Poisson. MAPMO-UMR 6628,
D\'epartement de Math\'ematiques, Universit\'e d'Orl\'eans, 45067
Orl\'eans Cedex 2, France}
\address[S. Revesz]{
 R\'enyi Institute of Mathematics, Hungarian Academy of Sciences,
 Budapest, P.O.B. 127, 1364 Hungary.}

\begin{abstract}
\selectlanguage{english}

We say that Wiener's property holds for the exponent $p>0$ if we
have that whenever a positive definite function $f$ belongs to
$L^p(-\varepsilon,\varepsilon)$ for some $\varepsilon>0$, then $f$
necessarily belongs to $L^p(\TT)$, too. This holds true for $p\in
2\NN$ by a classical result of Wiener.

Recently various concentration results were proved for idempotents
and positive definite functions on measurable sets on the torus.
These new results enable us to prove a sharp version of the
failure of Wiener's property for $p\notin 2\NN$. Thus we obtain
strong extensions of results of Wainger and Shapiro, who proved
the negative answer to Wiener's problem for $p\notin 2\NN$.

\medskip

\selectlanguage{francais}


  \noindent{\bf Contre-exemples \`a la
propri\'et\'e de Wiener pour les fonctions p\'eriodi--ques
d\'efinies-positives.}

 \noindent{\bf R\'esum\'e.} On dit que l'exposant $p$
poss\`ede la propri\'et\'e de Wiener si toute fonction
p\'eriodique d\'efinie-positive qui est de puissance $p$-i\`eme
int\'egrable au voisinage de $0$ l'est sur un intervalle de
p\'eriode. C'est le cas des entiers pairs, d'apr\`{e}s un r\'{e}sultat
classique de Wiener.

Nous avons r\'ecemment obtenu des ph\'enom\`enes de concentration
des polyn\^omes idempotents ou d\'efinis-positifs sur un ensemble
mesurable du tore qui nous permettent de donner une version forte
du fait que les exposants $p\notin 2\NN$ n'ont pas la
propri\'et\'e de Wiener, am\'eliorant ainsi les r\'esultats de
Wainger et Shapiro.

\end{abstract}


\selectlanguage{english}

\maketitle

\section{Introduction}\label{sec:intro}

Let $f$ be a periodic integrable function which is positive
definite, that is, has non negative Fourier coefficients. Assume
that it is bounded (in $\|\cdot\|_\infty$) in a neighborhood of
$0$, then it necessarily belongs to $L_\infty(\TT)$, too. In fact,
its maximum is obtained at $0$ and, as $f(0)=\sum_k
\widehat{f}(k)$,  $f$ has an absolutely convergent Fourier series.

The same question can be formulated in any $L^p$ space. Actually,
the following question was posed by Wiener in a lecture, after he
proved the $L^2$ case. We refer to \cite{Sh} for the story of this
conjecture, see also \cite{L}, \cite{Sh} and \cite{W}.

\begin{problem}[Wiener]\label{Wienerproblem} Let $1\le p<\infty$.
Is it true, that if for some $\varepsilon>0$ a positive definite
function $f\in L^p(-\varepsilon,\varepsilon)$, then we necessarily
have $f\in L^p(\TT)$, too?
\end{problem}

The observation that the answer is positive if $p\in 2\NN$ has
been given by Wainger \cite{Wa}, as well as by Erd\H os and Fuchs
\cite{EF}. We refer to Shapiro \cite{Sh} for the proof, since the
constant given by his proof is in some sense optimal, see
\cite{L,L2}. Generalizations in higher dimension may be found in
\cite {Hl} for instance. It was shown by Shapiro \cite{Sh} and
Wainger \cite{Wa} that the answer is to the negative for all other
values of $p$. Negative results were obtained for groups in e.g.
\cite{F} and \cite{L}.

There is even more evidence that the Wiener property must hold
when $p=2$ and we prescribe large gaps in the Fourier series of
$f$. Indeed, in this case by well-known results of Wiener and
Ingham, see e.g. \cite{W,Z}, we necessarily have an essentially
uniform distribution of the $L^2$ norm on intervals longer than
the reciprocal of the gap, even without the assumption that $f$ be
positive definite. As Zygmund pointed out, see the Notes to
Chapter V  \S 9, page 380 in \cite{Z}, Ingham type theorems were
not known for $p\ne 2$, nevertheless, one would feel that
prescribing large gaps in the Fourier series should lead to better
control of the global behavior by means of having control on some
subset like e.g. $(-\varepsilon,\varepsilon)$. So the analogous
Wiener question can be posed restricting to positive definite
functions having gaps tending to $\infty$. However, we answer
negatively as well. In this strong form the question, to the best
of our knowledge, has not been dealt with yet. Also we are able to
replace the interval $(-\varepsilon, +\varepsilon)$ by any
measurable symmetric subset $E$ of the torus of measure $|E|<1$.
Neither extension can be obtained by a straightforward use of the
methods of Shapiro and Wainger.

\section{$L^2$ results and concentration of integrals}\label{sec:concentration}

We use the notation $\TT:=\RR/\ZZ$ for the torus. Then
$e(t):=e^{2\pi i t}$ is the usual exponential function adjusted to
interval length $1$, and we  denote $e_h$ the function $e(hx)$.
The set of positive definite trigonometrical polynomials is the
set
\begin{equation}\label{eq:posdefpol}
\TP:=\left\{ \sum_{h\in H}a_k e_k ~:~ H\subset \ZZ \quad
(\textrm{or}~~ \NN), ~~ \# H< \infty, \quad a_k\geq 0 ~(k\in H)
\right\}
\end{equation}
For obvious reasons of being convolution idempotents, the set
\begin{equation}\label{eq:idempotents}
\PP:=\left\{ \sum_{h\in H}e_h ~:~ H\subset \ZZ \quad (\textrm{or}
~~\NN), ~~ \# H< \infty \right\}
\end{equation}
is called the set of \emph{(convolution-)idempotent exponential
(or trigonometric) polynomials}, or just \emph{idempotents} for
short.

Note that multiplying a polynomial by an exponential $e_K$ does
not change its absolute value, and the property of belonging to
$\PP$ or $\TP$ is not changed either. Therefore, it suffices to
consider polynomials with nonnegative spectrum, i.e. $H\subset
\NN$ in \eqref{eq:posdefpol} and \eqref{eq:idempotents}.

Also note that  for a positive definite function the function
$|f|$ is necessarily even. This is why we consider $0$-symmetric
(or, just symmetric for short) intervals or sets, (alternatively,
we could have chosen to restrict to $[0,1/2)$ instead of $\TT$).

Let us first state the theorem on positive definite functions in
$L^2$. Recall that the direct part is attributed to Wiener, with
the constant given by Shapiro in \cite{Sh}. The converse seems to
be well known (see \cite{L,L2}), except, may be, for the fact that
counter-examples may be given by idempotents. The fact that the
Wiener property fails for arbitrary measurable sets is, to the
best of our knowledge, new.

\begin{theorem}[Wiener, Shapiro] \label{th:shapiro} For $p$ an
even integer, for $0<a<1/2$ and for $f\in \TP$, we have the
inequality
\begin{equation}\label{shapiro}
\frac 1{2a}\int_{-a}^{+a}|f|^p\geq \frac 12
\int_{-1/2}^{+1/2}|f|^p.
\end{equation}
Moreover, the constant $1/2$ cannot be replaced by a smaller one,
even when restricting to idempotents. Indeed, for each integer
$k>2$, for $a<1/k$ and for $b>1/k$, there exits an idempotent $f$
 and such that $\int_{-a}^{+a}|f|^p\leq
b\times \int_{-1/2}^{+1/2}|f|^p$.
\end{theorem}
\begin{proof} We refer to Shapiro for the proof of the
inequality \eqref{shapiro}.

To show sharpness of the constant, let us now give an example,
inspired by the examples of \cite{DPQ}. We take $f:=D_n*\mu_k$,
where $D_n$ is the Dirichlet kernel, defined here as
\begin{equation}\label{eq:Dndef}
D_n(x):=\sum_{\nu=0}^{n-1} e(\nu x) = e^{\pi i(n-1)x/2}
\frac{\sin(\pi n x)}{\sin(\pi x)},
\end{equation}
and $\mu_k$ is the mean of Dirac masses at each $k$-th root of
unity. Both have Fourier coefficients $0$ or $1$, so that $f$ is
an idempotent. Only one of the point masses of $\mu_k$ lies inside
the interval $(-a,+a)$ and one can see that the ratio between
$\int_{-a}^{+a}|f|^p$  and $\int_{-1/2}^{+1/2}|f|^p$ tends to
$1/k$ when $n$ tends to infinity.
\end{proof}
\begin{remark} The interval $(-a,+a)$ cannot be replaced by a
measurable set $E$ having $0$ as a density point, even if $|E|$ is
arbitrarily close to $1$. Indeed, assume that the complement of
$E$ is the union (modulo $1$) of all intervals of radius $1/l^3$
around all irreducible rational numbers $k/l$, with $k$ different
from $0$ and $l>L$. Then $E$ has the required properties, while,
 for the same idempotent $f:=D_n*\mu_l$, the ratio between $\int_E|f|^p$  and
 $\int_{-1/2}^{+1/2}|f|^p$ tends to $1/l$ when $n$ tends to infinity.
 We get our conclusion noting that $l$ may be arbitrarily
 large.
 \end{remark}

Let us now consider the $p$-concentration problem, which comes
from  the following definition.

\begin{definition}
Let $p>0$, and $\mathcal F$ be a class of functions on $\TT$. We
say that for the class $\mathcal F$ there is $p$-concentration if
there exists a constant $c>0$ so that for any symmetric measurable
set $E$ of positive measure one can find an idempotent
$f\in{\mathcal F}$ with
\begin{equation}\label{eq:Lpconcentration}
\int_E |f|^p \geq c \int_\TT |f|^p.
\end{equation}
\end{definition}

The problem of $p$-concentration on the torus for idempotent
polynomials has been considered in \cite{DPQ}, \cite{DPQ2},
\cite{CRMany}. It was essentially solved recently in \cite{BR}.
Also, the weaker question of concentration of $p^{\textrm th}$
integrals of positive definite functions has been dealt with
starting with the works \cite{DPQ,DPQ2}. In this respect we have
proved the following result, see \cite[Theorem 48]{BR}. We will
only state that part of the theorems of \cite{BR} that we will
use.

\begin{theorem}\label{th:concentration} For all $0<p<\infty$, $p$
not an even integer,  whenever a $0$-symmetric measurable set $E$
of positive measure $|E|>0$ is given, then to all $\varepsilon>0$
there exists some positive definite trigonometric polynomial
$f\in\TP$
so that
\begin{equation}\label{eq:concentration}
\int_{^cE} |f|^p \leq\varepsilon \int_\TT |f|^p.
\end{equation}
Moreover, $f$ can be taken with arbitrarily large prescribed gaps
between frequencies of its Fourier series.
\end{theorem}

\begin{remark} The same result is also proved for open
symmetric sets and idempotents, and for measurable sets and
idempotents when $p>1$.
\end{remark}

Theorem \ref{th:concentration} allows to see immediately that
there is no inequality like \eqref{shapiro} for $p$ not an even
integer. What is new, compared to the results of Shapiro and
Wainger, is the fact that this is also the case if $f$ has
arbitrarily large gaps, and that we can replace intervals
$(-a,+a)$ by arbitrary measurable sets of measure less than $1$.
We will give a different statement in the next section for $E$ an
open set, and also show a strong version of the negative state of
Wiener's problem.

\section{Negative results in Wiener's problem}\label{sec:results}

Let us start with somewhat strengthening the previous theorem for
open sets, which we obtain by an improvement of the methods of
Shapiro in \cite{Sh}.
\begin{theorem}\label{th:strong-conc} For all $0<q\leq p<2$,  whenever a $0$-symmetric open set $E$
of positive measure $|E|>0$ is given, then for all $\varepsilon>0$
there exists some positive definite trigonometric polynomial
$f\in\TP$
so that
\begin{equation}\label{eq:strong-conc}
\int_{^cE} |f|^p \leq\varepsilon \left (\int_\TT
|f|^q\right)^{p/q}.
\end{equation}
The same is valid for $q<p$ with $p$ not an even integer, provided
that $q$ is sufficiently close to $p$, that is $q>q(p)$, where
$q(p)<p$.
\end{theorem}

The construction is closely related to the failure of Hardy
Littlewood majorant property. We do not know whether, for $p>2$
not an even integer, that is $2k<p<2k+2$,  we can take $q(p)=2k$.
Due to Theorem \ref{th:shapiro}, we cannot take $q(p)<2k$.  We do
not know either whether the next statement is valid for functions
with arbitrary large gaps.

\begin{proof}
Let us first assume that $p<2$. Then, for $D_n$ the Dirichlet
kernel with $n$ sufficiently large depending on $\varepsilon$,
there exists a choice of $\eta_k=\pm 1$  such that
$$\|D_n\|_p\leq \varepsilon \|\sum_{k=0}^n \eta_k e_k\|_q.$$
Indeed, if it was not the case, taking the $q$-th power,
integrating on all possible signs and using Khintchine's
Inequality, we would find that $c\varepsilon\sqrt n \leq
\|D_n\|_p\leq Cn^{1-\frac 1p}$ ($p>1$), $c\varepsilon\sqrt n \leq
\|D_n\|_1\leq C \log n$ and $c\varepsilon\sqrt n \leq
\|D_n\|_p\leq C$ ($0<p<1$) which leads to a contradiction.

We assume that $E$ contains $I\cup (-I)$, where $I:=(\frac kN,
\frac{k+1}N)$, and denote $$g(t):=\sum_{k=0}^n \eta_k
e_k(t)\hspace{2cm} G(t):=D_n(t).$$ Let $\Delta$ be a triangular
function based on the interval $ (-\frac 1{2N}, +\frac{1}{2N})$,
that is, $\Delta(t):=\left (1-2N|t|\right)_+ $. We finally
consider the function
$$f(t):=\Delta(t-a)g(2Nt)+\Delta(t+a)g(2Nt)+2\Delta(t)G(2Nt),$$ where
$a$ is the center of the interval $I$. Then an elementary
computation of Fourier coefficients, using the fact that $\Delta$
has positive Fourier coefficients while the modulus of those of
$g$ and $G$ are equal, allows to see that $f$ is positive
definite. Let us prove that one has (\ref{eq:strong-conc}). The
left hand side is bounded by $\frac 2N \|G\|_p^p$, while $
\int_\TT |f|^q $ is bounded below by $\frac 1{2N}\|g\|_q^q- \frac
2N \|G\|_q^q$. We conclude the proof choosing $n, N$ sufficiently
large.

\medskip

Let us now consider $p>2$ not an even integer. Mockenhaupt and
Schlag in \cite{MS} have given counter-examples to the Hardy
Littlewood majorant conjecture, which are based on the following
property: for $j>p/2$ an odd integer, the two trigonometric
polynomials
$$g_0:=(1+e_j)(1- e_{j+1})\hspace{2cm} G_0:=(1+e_j)(1+ e_{j+1})$$
satisfy the inequality $\|G_0\|_p<\|g_0\|_p$. By continuity, this
inequality remains valid when $p$ is replaced by $q$ in the right
hand side, with $q>q(p)$, for some $q(p)<p$. By a standard Riesz
product  argument, for $K$ large enough, as well as $N_1,
N_2,\cdots N_K$, depending on $\varepsilon$, the functions
$$g(t):=g_0(t)g_0(N_1t)\cdots g_0(N_Kt)\ \ \mbox{\rm and}\ \
G(t):=G_0(t)G _10(N_1t)\cdots G_0(N_Kt)$$ satisfy the inequality
$$\|G\|_p\leq \varepsilon \|g\|_q.$$
From this point the proof is identical.
\end{proof}

We can now state in two theorems the counter-examples that we
obtain for the Wiener conjecture when $p$ is not an even integer.

\begin{theorem}\label{th:noWiener} Let $0<p<\infty$, and $p\notin
2\NN$. Then for any symmetric, measurable set $E\subset\TT$ with
$|E|>0$ and any $q<p$, there exists a function $f$ in the Hardy
space $H^q(\TT)$ with positive Fourier coefficients,  so that its
pointwise boundary value $f^*$  is in $L^p(^cE)$ while $f^*\notin
L^p(\TT)$. Moreover, $f$ can be chosen with gaps  tending to
$\infty$ in its Fourier series.
\end{theorem}

 Here
$H^q(\TT)$ denotes the space of periodic distributions $f$ whose
negative coefficients are zero, and such that the function $f_r$
are uniformly in $L^q(\TT)$ for $0<r<1$, where
$$f_r(t):=\sum_{n }\hat f(n)r^{|n|} e^{2i\pi n t}.$$
Moreover, the norm (or quasi-norm) of $f$ is given by
$$\|f\|_{H^q(\TT)}^q:=\sup_{0<r<1}\int_0^1|f_r|^q.$$
 It is well known that, for $f\in H^q(\TT)$, the functions $f_r$
have an a. e. limit $f^*$ for $r$ tending to $1$. The function
$f^*$, which we call  the pointwise boundary value, belongs to
$L^q(\TT)$. When $q\geq 1$, then $f$ is the distribution defined
by $f^*$, and $H^q(\TT)$ coincides with the subspace of functions
in $L^q(\TT)$ whose negative coefficients are zero. In all cases
the space $H^q(\TT)$ identifies with the classical Hardy space
when  identifying the distribution $f$ with the holomorphic
function $\sum_{n\geq 0 }\hat f(n)z^n$ on the unit disc. This
explains the use of the term of boundary value.

The function $f\in H^q$ is said to have gaps (in its Fourier
series) tending to $\infty$ whenever its Fourier series of $f$ can
be written as $\sum_{k=0}^\infty a_k e^{2i\pi n_k x},$ where $n_k$
is an increasing sequence such that $n_{k+1}-n_k\to \infty$ with
$k$.

\smallskip

In opposite to this theorem, recall that for $n_k$ a
\emph{lacunary} series, if the Fourier series is in $L^p(E)$ for
some measurable set $E$ of positive measure, then the function $f$
belongs to all spaces $L^q(\TT)$, see \cite{Z}. This has been
generalized by Miheev \cite{M} to $\Lambda(p)$ sets for $p>2$: if
$f$ is in $L^p(E)$, then $f$ is in the space $L^p(\TT)$. See also
the expository paper \cite{BD}.
\smallskip


\begin{proof}
The key of the proof is  Theorem \ref{th:concentration}. Remark
that we can assume that $p>q>1$. Indeed, $f^\ell$ is a positive
definite function when $f$ is, and counter-examples for some $p>1$
will lead to counter-examples for $p/\ell$. Now, let us take a
sequence $E_k$ of disjoint measurable subsets of $E$ of positive
measure, such that $|E_k|<2^{-\alpha k}$, with $\alpha$ to be
given later and let $f_k$ be a sequence of positive definite
trigonometric polynomials such that
\begin{equation}\label{first}
 \int_{ \TT\setminus E_k} |f_k|^p \leq 2^{-kp } \int_{\TT} |f_k|^p .
\end{equation}
Moreover, we assume that $f_k$'s have gaps larger than $k$. Using
H\"older's inequality, we obtain
\begin{align*}\label{eq:smallernorm}
\int_{\TT}|f_k|^q \leq 2^{-\alpha (1-q/p) k}\left(\int_{E_k}
|f_k|^p\right)^{q/p}+\left(\int_{\TT\setminus E_k}
|f_k|^p\right)^{q/p} \leq  2\times 2^{-{kq}}\left (\int_{\TT}
|f_k|^p\right)^{q/p},
\end{align*}
if $\alpha$ is chosen large enough. Finally, we normalize the
sequence $f_k$ so that $\int_{\TT} |f_k|^p=2^{\frac k{2}}$, and
take
\begin{equation}\label{series}
  f(x):=\sum_{k\geq 1} e^{2i\pi m_k x}f_k(x),
\end{equation}
where the $m_k$ are chosen inductively sufficiently increasing, so
that  the condition on gaps is satisfied. The series is convergent
in $L^q(\TT)$ and in $L^p(^cE)$, and the limit $f$ has its Fourier
series given by  \eqref{series}. Now, let us  prove that $f$ is
not in $L^p(\TT)$. Since the $E_j$'s are disjoint,
$$
\| f\|_{p}\geq \| f\|_{L^p(E_k)} \geq \| f_k\|_{p} - \sum_{j} \|
f_j\|_{L^p(^cE_j)} \geq 2^{\frac k2} - \sum_{j>0} 2^{-\frac j2},
$$ which allows to conclude.
\end{proof}

Using Theorem \ref{th:strong-conc} instead of Theorem
\ref{th:concentration}, we have  the following.
\begin{theorem}\label{th:strong-noWiener}\begin{itemize}\item[(i)]
Let $p>2$, with $p\notin 2\NN$, and let $\ell\in\NN$ such that
$2\ell<p<2(\ell +1)$. Then, for any symmetric open set
$U\subset\TT$ with $|U|>0$ and $q>q(p)$, there exists a positive
definite function $f\in L^{2\ell}(\TT)$, whose negative
coefficients are zero,  such that $f\notin L^q(\TT)$ while $f$ is
in $L^p(^cU)$.
\item[(ii)]
Let $0<p<2$. Then
 for any symmetric open set $U\subset\TT$ with $|U|>0$
and any $s<q<p$, there exists a function $f$ in the Hardy space
$H^{s}(\TT)$ with non negative Fourier coefficients,  so that
$f\notin H^q(\TT)$ while  $f^*$ is  in $L^p(^cU)$.
\end{itemize}
\end{theorem}

 \begin{proof} Let us first prove $(i)$. We can assume that $^cU$ contains a neighborhood of $0$. So, by Wiener's property,
if $f$ is integrable and belongs to $L^p(^cU)$, then $f$ is in
$L^{2\ell}(\TT)$. Let us prove that there exists such a function,
whose Fourier coefficients satisfy the required properties, and
which does not belong to $L^q(\TT)$. The proof follows the same
lines as in the previous one.  By using Theorem
\ref{th:strong-conc}, we can find positive definite polynomials
$f_k$ such that $\|f_k\|_q=2^{k/2}\to \infty$, while $\|
f_k\|_{L^p(^cU_k)}\leq 2^{-k}$ with $U_k\subset U$ disjoint and of
sufficiently small measure, so that
$\sum\|f_k\|_{L^p(^cU)}<\infty$. As before, the function $
f:=\sum_{k\geq 1} e_{ m_k}f_k$ will have the required properties.

Let us now consider $1\leq p<2$, from which we conclude for
$(ii)$: if $p< 1$, we look for a function of the form $f^{\ell}$,
with $f$ satisfying the conclusions for $\ell p$, with $\ell$ such
that $1\leq \ell p<2$. We can assume that $q< 1$. We proceed as
before, with  $f_k$'s  given by Theorem \ref{th:strong-conc}, such
that   $\|f_k\|_q=2^{k/2}$ and $\| f_k\|_{L^p(^cU_k)}\leq
2^{-k/2}$. The $U_k$'s are assumed to be disjoint and of small
measure, so that $\sum_k \|f_k\|_{H^s}^s<\infty$. It follows
 that $f\in H^s(\TT)$. Remark that
$f$ is not a function, in general, but a distribution. Recall that
$f^*$ is the boundary value of the corresponding holomorphic
function. We write as before
$$
\| f\|_{H^q(\TT)}^q\geq \| f^*\|_{L^q(U_k)}^q \geq \| f_k\|_{q}^q
- \sum_{j} \| f_j\|_{L^q(^cU_j)}^q \geq 2^{\frac {kq}2} -
\sum_{j>0} 2^{-\frac {jq}2},
$$ which allows to conclude for the fact that $f$ is not in
$H^q(\TT)$.

\begin{remark} As Wainger in \cite{W}, we can prove a little more:
the function $f$ may be chosen such that  $\sup_{r<1}|f_r|$ is in
$L^p(^cU)$. Let us give the proof in the case $(i)$. We can assume
that $U$ may be written as $I\cup(-I) $
 for some interval $I$. Let $J$ be the interval of same center and
 length half, and take $f$ constructed as wished, but for the open
 set $J\cup(-J)$. Finally, write $f=\phi+\psi$, with $\phi:=f\chi_{^c\left(J\cup
 (-J)\right)}$. Then using the maximal theorem we know that $\sup_{r<1}|\phi_r|\in
 L^p(\TT)$, while the Poisson kernel $P_t(x-y)$ is uniformly bounded
 for $x\notin U$ and $y\in J\cup(-J)$, so that $\sup_{r<1}|\psi_r|$ is uniformly bounded outside $U$.

 In the case $(ii)$, the proof is more technical, $f$ being only a
 distribution. We use the fact that derivatives
 of the Poisson kernel $P_t(x-y)$ are also uniformly bounded
 for $x\notin U$ and $y\in J\cup(-J)$.
 \end{remark}

\end{proof}

\end{document}